\documentclass[10pt]{amsart}
\usepackage[english]{babel}
\usepackage{amsmath,amsthm,amssymb}
\usepackage{amsfonts}
\usepackage{enumerate}

\usepackage{graphicx}
\usepackage{color}
\usepackage[top=1.41in, bottom=1.41in, left=1.75in, right=1.75in]{geometry}

\usepackage[all]{xy}

\newtheorem{thm}{Theorem}[section]
\newtheorem{cor}[thm]{Corollary}
\newtheorem{lem}[thm]{Lemma}
\newtheorem{prop}[thm]{Proposition}
\newtheorem{conj}{Conjecture}
\theoremstyle{definition}
\newtheorem{defn}[thm]{Definition}
\newtheorem{ques}[conj]{Question}
\theoremstyle{remark}
\newtheorem{rem}[thm]{Remark}
\newtheorem{exam}[thm]{Example}
\numberwithin{equation}{section}
\newcommand{\norm}[1]{\lVert #1 \rVert^2}

\newcommand{\spin}{\ifmmode{\rm Spin}\else{${\rm spin}$\ }\fi}
\newcommand{\spinc}{\ifmmode{{\rm Spin}^c}\else{${\rm spin}^c$}\fi}
\newcommand{\spinct}{\mathfrak t}
\newcommand{\spincs}{\mathfrak s}
\newcommand{\Char}{{\rm Char}}
\newcommand{\rk}{{\rm rk }}

\begin{document}

\title{Bounds on alternating surgery slopes}%
\author{Duncan McCoy}%
%\address{University of Texas at Austin}%
%\email{d.mccoy.1@research.gla.ac.uk}
\date{}%

\begin{abstract}
We show that if $p/q$-surgery on a nontrivial knot $K$ yields the branched double cover of an alternating knot, then $|p/q|\leq 4g(K)+3$. This generalises a bound for lens space surgeries first established by Rasmussen. We also show that all surgery coefficients yielding the double branched covers of alternating knots must be contained in an interval of width two and this full range can be realised only if the knot is a cable knot. The work of Greene and Gibbons shows that if $S^3_{p/q}(K)$ bounds a sharp 4-manifold $X$, then the intersection form of $X$ takes the form of a changemaker lattice. We extend this to show that the intersection form is determined uniquely by the knot $K$, the slope $p/q$ and the Betti number $b_2(X)$.
\end{abstract}
% ----------------------------------------------------------------

\maketitle

\section{Introduction}
For a knot $K\subset S^3$ and $p/q\in \mathbb{Q}$ we say that $S_{p/q}^3(K)$ is an {\em alternating surgery} if it is the double branched cover of an alternating knot or link. In this paper, we will prove some bounds on the slopes of alternating surgeries. The first of these generalises a bound for lens space surgeries originally due to Rasmussen \cite{Rasmussen04Goda}.
\begin{thm}\label{thm:upperbound}
If $K$ is a nontrivial knot with an alternating surgery $S_{p/q}^3(K)$, then the slope $p/q$ satisfies the inequality $|p/q| \leq 4g(K)+3$.
\end{thm}
The bound in Theorem~\ref{thm:upperbound} is sharp with equality being attained by the $T_{2,n}$ torus knots. It turns out that whenever this bound is realised the resulting alternating surgery yields a lens space. Hence, work of Baker shows that the $T_{2,n}$ torus knots are the only knots achieving achieving equality in Theorem~\ref{thm:upperbound} \cite[Theorem~1.2]{baker06smallgenus}.

We can also obtain a bound on the range of slopes yielding alternating surgeries.
\begin{thm}\label{thm:widthbound}
If $K$ is a nontrivial knot admitting an alternating surgery, then there is an integer $N$, such that for any alternating surgery $S_{p/q}^3(K)$, the coefficient $p/q$ lies in the interval \[N-1\leq p/q \leq N+1.\]
\end{thm}
The definition of $N$ is given Section~\ref{subsec:mainresults}. Theorem~\ref{thm:widthbound} shows that the range of slopes which yield alternating surgeries is contained in an interval with integer endpoints of width two. When every slope in this interval yields an alternating surgery, then we will show that the knot must be a cable knot. For the purposes of this paper, we consider torus knots to be cable knots.
\begin{thm}\label{thm:fullrange}
Suppose that $K$ is a nontrivial knot admitting alternating surgeries $S_{r}^3(K)$ for each of the slopes $r\in \{r_1,r_2,N\}$, where $N$ is the integer appearing in Theorem~\ref{thm:widthbound}. If $r_1$ and $r_2$ satisfy
\[N-1\leq r_1<N<r_2< N+1,\]
then $S_{N}^3(K)$ is a reducible surgery and $K$ is a cable knot.
\end{thm}
\begin{rem}\label{rem:fullrangeextension}
It can be shown that Theorem~\ref{thm:fullrange} still holds under the slightly weaker condition that $r_2\leq N+1$. However, this relatively minor extension requires a substantial amount of work so we will not prove it here.
\end{rem}

The starting point for the proof of these results is the work of Gibbons \cite{gibbons2013deficiency}, which generalizes the work of Greene \cite{GreeneLRP,Greene3Braid,greene2010space}. It provides strong restrictions on the intersection form of a negative-definite sharp $4$-manifold $X$ bounding $S^3_{p/q}(K)$ for $p/q>0$, which must take the form of a changemaker lattice. In order to prove Theorem~\ref{thm:widthbound} and Theorem~\ref{thm:fullrange} we are required to determine the extent to which this intersection form depends on the knot $K$ and the surgery slope $p/q$. This leads us to define the stable coefficients of a changemaker lattice. The definition of a changemaker lattice and its stable coefficients are given in Section~\ref{sec:cmlattices}. Let $p/q$ have continued fraction expansion $p/q=[a_0, \dots, a_l]^-$, where $a_i\geq 2$ for $1\leq i \leq l$ and $a_0\geq 1$. Here $[a_0, \dotsc, a_l]^-$ denotes the Hirzebruch-Jung continued fraction:
\[[a_0, \dotsc, a_l]^-
    = a_0 - \cfrac{1}{a_1
           - \cfrac{1}{\ddots
           - \cfrac{1}{a_l} } }.\]
A $p/q$-changemaker lattice takes the form of an  orthogonal complement:
\[L=\langle w_0, \dotsc , w_l \rangle^\bot \subseteq \mathbb{Z}^{t+s+1}=\langle f_1, \dotsc , f_t, e_0, \dotsc , e_s \rangle,\]
where the $f_i$ and $e_j$ form an orthonormal basis for $\mathbb{Z}^{r+s+1}$, and the $w_i$ have the properties that
\[w_i \cdot w_j =
    \begin{cases}
        a_i   & i=j\\
        -1   & |i-j|=1 \\
        0   &  |i-j|\geq 2,
    \end{cases}
\]
and
\[w_0\cdot e_0=1, \quad w_0\cdot e_i = 0 \quad \text{for } 1\leq i \leq s,\]
\[w_0\cdot f_i \geq 0 \quad \text{for } 1\leq i \leq t,\]
\[w_j\cdot f_i = 0 \quad \text{for } 1\leq i \leq t\text{ and } 1\leq j \leq l.\]
The stable coefficients of $L$ are defined to be the values of $w_0\cdot f_i$ satisfying $w_0\cdot f_i>1$.
\begin{thm}\label{thm:Gibbonssouped}
Let $K \subset S^3$ be a knot and suppose that for some $p/q>0$, $S^3_{p/q}(K)$ bounds a negative-definite sharp 4-manifold $X$ with intersection form $Q_X$. Then the positive-definite lattice $-Q_X$ embeds into $\mathbb{Z}^{b_2(X)+l+1}$ as a $p/q$-changemaker lattice, where the stable coefficients are determined by $K$.
\end{thm}
The stable coefficients in Theorem~\ref{thm:Gibbonssouped} form an invariant of the knot $K$ that can be calculated from the knot Floer homology of $K$. Section~\ref{sec:CMinvariant} provides an algorithm for this calculation. When $K$ is an $L$-space knot, the stable coefficients can be computed directly from its Alexander polynomial. The integer $N$ appearing in Theorem~\ref{thm:widthbound} and Theorem~\ref{thm:fullrange} is defined in terms of stable coefficients and hence is an invariant of $K$ and can be calculated from the Alexander polynomial.
\begin{rem}\label{rem:sharplowerbound}
In addition to being a lower bound for alternating surgeries, the integer $N-1$ appearing in Theorem~\ref{thm:widthbound} also has the property that if $S^3_{p/q}(K)$ bounds a negative-definite sharp 4-manifold then $p/q\geq N-1$. We explain this observation after the proof of Theorem~\ref{thm:widthbound}.
\end{rem}
Given one negative-definite sharp 4-manifold, bounding a 3-manifold $Y$ we can obtain another by taking a connected sum with $\overline{\mathbb{CP}}^2$. It follows from Theorem~\ref{thm:Gibbonssouped} that if $Y=S^3_{p/q}(K)$, then at the level of intersection forms this is the only possibility.
\begin{cor}\label{cor:interfromunique}
Let $K \subset S^3$ be a knot such that for some $p/q>0$, the 3-manifold $S^3_{p/q}(K)$ bounds negative-definite sharp 4-manifolds $X$ and $X'$, with $b_2(X')= b_2(X)+k$ for $k\geq 0$. Then
\[ Q_{X'} \cong Q_{X} \oplus (-\mathbb{Z}^{k})\cong Q_{X\#_{k}\overline{\mathbb{CP}}^2}.\]
\end{cor}

\subsection*{Acknowledgements}
The author would like to thank his supervisor, Brendan Owens, for his guidance and careful reading of this paper. He is grateful to Liam Watson for helpful conversations about quasi-alternating links and many other things. He also wishes to thank the anonymous referee for their feedback.

\section{Changemaker lattices and sharp 4-manifolds}
The aim of this section is to prove Theorem~\ref{thm:Gibbonssouped}. We begin by defining changemaker lattices and recalling the necessary definitions and properties from Heegaard Floer homology. We finish the section by stating the properties of $L$-space surgeries that we will require to prove the results on alternating surgeries.

\subsection{Changemaker lattices}\label{sec:cmlattices}
We will define $p/q$-changemaker lattices for any $p/q>0$. Changemaker lattices corresponding to the case $q=1$ were defined by Greene in his solution to the lens space realization problem \cite{GreeneLRP} and work on the cabling conjecture \cite{greene2010space}. The case $q=2$ arose in his work on unknotting numbers \cite{Greene3Braid}. The more general definition we state here is the one which arises in Gibbons' work \cite{gibbons2013deficiency}.

\begin{defn}We say $(\sigma_1, \dots , \sigma_t)$ satisfies the {\em changemaker condition}, if the following conditions hold,
$$0\leq \sigma_1 \leq 1, \text{ and } \sigma_{i-1} \leq \sigma_i \leq \sigma_1 + \dots + \sigma_{i-1} +1,\text{ for } 1<i\leq t.$$
\end{defn}

The changemaker condition is equivalent to the following combinatorial result.
\begin{prop}[Brown \cite{brown1961note}]\label{prop:CMprop}
Let  $\sigma =(\sigma_1, \dots , \sigma_t)$, with $\sigma_1\leq \dots \leq \sigma_t$. There is $A\subseteq \{ 1, \dots , t\}$ such that $k=\sum_{i\in A} \sigma_i$, for every integer  $k$ with $0\leq k\leq \sigma_1 + \dots + \sigma_t$, if and only if $\sigma$ satisfies the changemaker condition.
\end{prop}

Now we are ready to define changemaker lattices. It is convenient to define integer and non-integer changemaker lattices separately, although the two are clearly similar.

\begin{defn}[Integral changemaker lattice]\label{def:intCMlattice}
First suppose that $q=1$, so that $p/q>0$ is an integer. Let $f_0, \dotsc, f_t$ be an orthonormal basis for $\mathbb{Z}^t$. Let $w_0=\sigma_1 f_1 + \dotsb + \sigma_t f_t$ be a vector such that $\norm{w_0}=p$ and $(\sigma_1, \dotsb, \sigma_t)$ satisfies the changemaker condition. Then
\[L=\langle w_0\rangle^\bot \subseteq \mathbb{Z}^{t+1}\]
is a {\em $p/q$-changemaker lattice}. Let $m$ be minimal such that $\sigma_m>1$. We define the {\em stable coefficients} of $L$ to be the tuple $(\sigma_m, \dots, \sigma_t)$. If no such $m$ exists, then we take the stable coefficients to be the empty tuple.
\end{defn}

\begin{defn}[Non-integral changemaker lattice]\label{def:nonintCMlattice}
Now suppose that $q\geq 2$ so that $p/q>0$ is not an integer. This has continued fraction expansion of the form $p/q=[a_0,a_1, \dots , a_l]^{-}$, where $a_k\geq 2$ for $1\leq k \leq l$ and $a_0=\lceil \frac{p}{q}\rceil \geq 1$. Now define
\[m_0=0 \text{ and } m_k=\sum_{i=1}^ka_i -k \text{ for } 1\leq k \leq l.\]
Set $s=m_{l}$ and let $f_1, \dotsc, f_t, e_0, \dotsc, e_s$ be an orthonormal basis for the lattice $\mathbb{Z}^{t+s+1}$.
Let $w_0=e_0+\sigma_1 f_1 + \dotsb + \sigma_t f_t,$ be a vector such that $(\sigma_1, \dots, \sigma_t)$ satisfies the changemaker condition and $\norm{w_0}=a_0$. For $1\leq k \leq l$, define
\[w_k=-e_{m_{k-1}}+e_{m_{k-1}+1}+ \dotsb + e_{m_{k}}.\]
We say that
\[L=\langle w_0, \dotsc, w_l\rangle^\bot \subseteq \mathbb{Z}^{t+s+1}\]
is a {\em $p/q$-changemaker lattice}. Let $m$ be minimal such that $\sigma_m>1$. We define the {\em stable coefficients} of $L$ to be the tuple $(\sigma_m, \dots, \sigma_t)$. If no such $m$ exists, then we take the stable coefficients to be the empty tuple.
\end{defn}

\begin{rem}Since $m_k-m_{k-1}=a_k-1$, the vectors $w_0, \dotsc, w_l$ constructed in Definition~\ref{def:nonintCMlattice} satisfy
\[
w_i.w_j =
  \begin{cases}
   a_j            & \text{if } i=j\\
   -1       & \text{if } |i-j|=1\\
   0        & \text{otherwise.}
  \end{cases}
\]
\end{rem}
\begin{rem}\label{rem:CMdetermined}
Let $L$ be a $p/q$-changemaker lattice:
 \[L=\langle w_0=e_0+\sigma_1 f_1 + \dotsb + \sigma_t f_t, w_1, \dotsc, w_l\rangle^\bot \subseteq \mathbb{Z}^{t+s+1}.\]
 By definition, the stable coefficients determine the values of the $\sigma_i$ satisfying $\sigma_i>1$. Since $\norm{w_0}=\lceil \frac{p}{q}\rceil$, the stable coefficients fix the number of $\sigma_i$ equal to 1 and this accounts for all non-zero $\sigma_i$. It follows that the number of $\sigma_i$ equal to zero can be deduced from the rank of $L$. Thus we see that the value $p/q$, the stable coefficients and the rank determine $L$ uniquely. Since we have $f_i\in L$ if and only if $\sigma_i=0$, any two $p/q$-changemaker lattices $L$ and $L'$ with the same stable coefficients and $\rk(L') = \rk(L) + k$ satisfy $L'\cong L \oplus \mathbb{Z}^k$
 \end{rem}

\subsection{Sharp 4-manifolds}\label{sec:sharp}
Now we will give a summary of the necessary background on Heegaard Floer homology and its $d$-invariants.
Let $Y$ be a rational homology 3-sphere. Its Heegaard Floer homology, $\widehat{HF}(Y)$, when defined with coefficients in $\mathbb{Z}/2\mathbb{Z}$, takes the form of a finite dimensional vector space over $\mathbb{Z}/2\mathbb{Z}$. The group $\widehat{HF}(Y)$ splits as a direct sum over \spinc-structures:
\[\widehat{HF}(Y)\cong \bigoplus_{\spincs \in \spinc(Y)}\widehat{HF}(Y,\spincs),\]
where $\widehat{HF}(Y,\spincs)\ne 0$ for all $\spincs \in \spinc(Y)$. We say that $Y$ is an {\em $L$-space} if $\widehat{HF}(Y)$ is as small as possible:
\[\dim_{\mathbb{F}_2} \widehat{HF}(Y) = |H^2(Y;\mathbb{Z})|=|\spinc(Y)|.\]
Associated to each summand there is a numerical invariant $d(Y,\spincs)\in \mathbb{Q}$, called the {\em $d$-invariant} \cite{Ozsvath03Absolutely}. If $Y$ is the boundary of a smooth negative-definite 4-manifold $X$, then for any $\spinct \in \spinc(X)$ which restricts to $\spincs \in \spinc(Y)$ there is a bound on the corresponding $d$-invariant:
\begin{equation}\label{eq:sharpdef}
c_1(\spinct)^2+b_2(X)\leq 4d(Y,\spincs).
\end{equation}
We say that $X$ is {\em sharp} if for every $\spincs \in \spinc(Y)$ there is some $\spinct \in \spinc(X)$ which restricts to $\spincs$ and attains equality in \eqref{eq:sharpdef}.

\paragraph{}We will be interested in the case where $Y$ arises as surgery on a knot in $S^3$. Let $K \subset S^3$ be a knot. For fixed $p/q \in \mathbb{Q}\setminus \{0\}$, there are canonical identifications \cite{ozsvath2011rationalsurgery}:
\[\spinc(S_{p/q}^3(K)) \leftrightarrow \mathbb{Z}/p\mathbb{Z} \leftrightarrow \spinc(S_{p/q}^3(U)).\]
Using these identifications we are able to define
\[D_{p/q}(i):=d(S_{p/q}^3(K),i)-d(S_{p/q}^3(U),i),\]
for each $i \in \mathbb{Z}/p\mathbb{Z}$.

\paragraph{}The work of Ni and Wu shows that for $0\leq i \leq p-1$ these values may be calculated by the formula \cite[Proposition 1.6]{ni2010cosmetic},
\begin{equation}\label{eqn:NiWuHV}
D_{p/q}(i)=-2\max\{V_{\lfloor \frac{i}{q} \rfloor},H_{\lfloor \frac{i-p}{q} \rfloor}\},
\end{equation}
where $V_j$ and $H_j$ are sequences of positive integers depending only on $K$, which are non-increasing and non-decreasing respectively. These further satisfy $H_{-j}=V_j=0$ for $j\geq g(K)$, where $g(K)$ is the genus of $K$. In fact, it can be shown that $V_j=H_{-j}$ for all $j$ \cite[Proof of Theorem~3]{owensstrle2013immersed}. Using these properties of the $V_j$ and $H_j$, \eqref{eqn:NiWuHV} can be rewritten as
\begin{equation}\label{eqn:NiWuVonly}
D_{p/q}(i)=-2V_{\min \{ \lfloor \frac{i}{q}\rfloor , \lceil \frac{p-i}{q}\rceil \}}.
\end{equation}
Let $p/q=[a_0, \dotsc, a_l]^-$ be the continued fraction of $p/q$ with $a_0\geq 1$ and $a_i\geq 2$ for $i\geq 1$. The changemaker theorem we will use is the following.

\begin{thm}[Gibbons \cite{gibbons2013deficiency}]\label{thm:Gibbons}
Let $K \subset S^3$ be a knot and suppose that for some $p/q>0$, $S^3_{p/q}(K)$ bounds a smooth, negative-definite 4-manifold $X$ with intersection form $Q_X$. If the manifold $X$ is sharp, then $-Q_X$ embeds into $\mathbb{Z}^{b_2(X)+l+1}$ as a $p/q$-changemaker lattice,
\[-Q_X\cong L=\langle w_0, \dotsc, w_l\rangle^\bot \subseteq \mathbb{Z}^{t+s+1},\]
where $w_0$ satisfies the formula:
\begin{equation}\label{eq:CMformula}
8V_{|i|} = \min_{\substack{ c\cdot w_0 \equiv a_0 + 2i \bmod 2a_0 \\ c \in \Char(\mathbb{Z}^{t+1})}} \norm{c} - t -1,
\end{equation}
for $|i|\leq a_0/2$.
\end{thm}
Here $\Char(\mathbb{Z}^{t+1})$ denotes the set of all characteristic vectors in $\mathbb{Z}^{t+1}$, where a {\em characteristic vector} $x\in \mathbb{Z}^{t+1}$ is one with odd coefficients with respect to any orthonormal basis for $\mathbb{Z}^{t+1}$.

The equation \eqref{eq:CMformula} is not explicitly stated by Gibbons. However, Greene shows that it holds in the case of integer surgeries \cite{greene2010space} and it follows from Gibbons' proof that it must also hold for non-integer surgeries. Further discussion of this can be found in \cite{mccoy2014alexpoly}.
 %Theorem~\ref{thm:Gibbons} omits the hypotheses on the $d$-invariants of $S^3_{p/q}(K)$ which were present in Gibbons' original statement, since it can be shown that they are automatically satisfied (cf. \cite[Section~2]{mccoy2014noninteger}). %It is also worth observing that we do not require $X$ to be simply-connected. The only way this condition was used in the proof of Theorem~\ref{thm:Gibbons} is to ensure that the restriction map $\spinc(X) \rightarrow \spinc(S^3_{p/q}(K))$ is surjective.

\subsection{Calculating stable coefficients}\label{sec:CMinvariant}
We will deduce Theorem~\ref{thm:Gibbonssouped} from Theorem~\ref{thm:Gibbons} by showing that \eqref{eq:CMformula} determines the stable coefficients uniquely. The argument is entirely combinatorial and uses only the properties of the $V_i$ stated in Section~\ref{sec:sharp}.

Let $(V_i)_{i\geq0}$ be the non-increasing, non-negative sequence
\[V_0\geq V_1 \geq \dotsb \geq V_{\tilde{g}-1}> V_{\tilde{g}}=V_{\tilde{g}+1}= \dotsb=0, \]
 for which $V_i=0$ if and only if $i\geq \tilde{g}$ and $V_i\leq V_{i+1}+1$ for all $i$. Suppose that there is $\rho=(\rho_0, \dotsc, \rho_t) \in \mathbb{Z}^{t+1}$, with $\norm{\rho}=n\geq 2\tilde{g}$ such that
\begin{equation}\label{eq:Vi1}
8V_{|k|} = \min_{\substack{c\cdot \rho \equiv n + 2k \bmod 2n \\c \in \Char(\mathbb{Z}^{t+1})}} \norm{c} - t -1,
\end{equation}
for $|k| \leq n/2$. Possibly after an automorphism of $\mathbb{Z}^{t+1}$, we may assume that $\rho_i\geq 0$ for all $i$ and that the $\rho_i$ form a decreasing sequence:
\[\rho_0\geq \rho_1\geq  \dotsb \geq \rho_t\geq 0.\]
Observe that \eqref{eq:Vi1} has three pieces of input data, the sequence $(V_i)_{i\geq 0}$ and the integers $n$ and $t$. Given some choice of $(V_i)_{i\geq 0}$, $n$ and $t$, there is no guarantee that there is $\rho$ satisfying \eqref{eq:Vi1}. However, we will show that when there is such a $\rho$, then it is unique. Moreover we will see that the coefficients of $\rho$ satisfying $\rho_i>1$ are determined by the sequence $(V_i)_{i\geq 0}$.
\begin{rem}\label{rem:removezeroes}
If $\rho_t=0$, then any minimiser in the right hand side of \eqref{eq:Vi1} must have $c_t=\pm 1$. So we see that $\rho'=(\rho_0, \dotsc, \rho_{t-1})$ satisfies
\begin{equation*}
8V_{|k|} = \min_{\substack{c\cdot \rho' \equiv n + 2k \bmod 2n \\c \in \Char(\mathbb{Z}^{t})}}\norm{c} - t,
\end{equation*}
for all $0\leq |k|\leq n/2$. This allows us to assume that $\rho_i\geq 1$ for all $i$.
\end{rem}
If we restrict our attention to $0\leq k \leq n/2$, we find that \eqref{eq:Vi1} simplifies as follows.
\begin{lem}\label{lem:modremoved}
For $0\leq k \leq n/2$,
\[8V_k = \min_{\substack{c\cdot \rho = 2k-n \\ c \in \Char(\mathbb{Z}^{t+1})}} \norm{c} - t -1.\]
\end{lem}
\begin{proof}
Suppose $c \in \Char(\mathbb{Z}^{t+1})$, satisfies $c\cdot \rho = 2mn-n + 2k$ for some $m\in \mathbb{Z}$. Consider the vector $c'=c-2m \rho$. This satisfies
\[c'\cdot \rho =2k-n \equiv c \cdot \rho \bmod 2n\]
and
\begin{align*}
\norm{c'}&= \norm{c} - 4m c \cdot \rho + 4m^2 n \\
         &= \norm{c} - 4m(nm-n+2k).
\end{align*}
Since we are assuming $-n\leq 2k-n\leq 0$, we have $m(nm+2k-n)\geq 0$ for all $m \in \mathbb{Z}$. Therefore, we have $\norm{c'}\leq \norm{c}$. This shows that if $c$ is a minimiser in \eqref{eq:Vi1} we can assume it satisfies $c\cdot \rho = 2k-n$.
\end{proof}

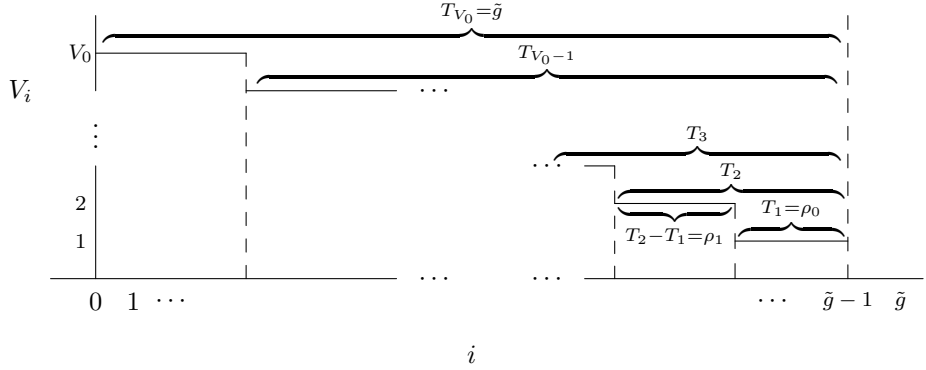
\begin{figure}[h]
\begin{xy}
<1cm,1cm>*{};<1cm,2.5cm>*{}**@{-},  %y-axis
<1cm,3.5cm>*{};<1cm,4.5cm>*{}**@{-},  %y-axis
<0.4cm,1cm>*{};<5cm,1cm>*{}**@{-},%x-axis
<5.5cm,1cm>*{\dotsb},             %x-axis
<7cm,1cm>*{\dotsb},             %x-axis
<7.5cm,1cm>*{};<12cm,1cm>*{}**@{-},  %x-axis
<1cm,3cm>*{\vdots},             %y-axis
<0cm,3.5cm>*{V_i},                %y-axis title
<6cm,0cm>*{i},                    %x-axis title
<1cm,4cm>*{};<3cm,4cm>*{}**@{-},    %horizontal line
<7.5cm,2.5cm>*{};<7.9cm,2.5cm>*{}**@{-},    %horizontal line
<3cm,1cm>*{};<3cm,4cm>*{}**@{--},   %vertical line
<3cm,3.5cm>*{};<5cm,3.5cm>*{}**@{-},%horizontal line
<5.5cm,3.5cm>*{\dotsb},             %horizontal line dots
<7cm,2.5cm>*{\dotsb},             %horizontal line dots
<7.9cm,2cm>*{};<9.5cm,2cm>*{}**@{-},  %horizontal line
<7.9cm,2.5cm>*{};<7.9cm,1cm>*{}**@{--},   %vertical line
<9.5cm,2cm>*{};<9.5cm,1cm>*{}**@{--},   %vertical line
<9.5cm,1.5cm>*{};<11cm,1.5cm>*{}**@{-},  %horizontal line
<11cm,4.5cm>*{};<11cm,1cm>*{}**@{--},   %vertical line
<1cm,0.7cm>*{0},        %x-axis label
<2cm,0.7cm>*{\dotsb},   %x-axis label
<10cm,0.7cm>*{\dotsb}, %x-axis label
<1.5cm,0.7cm>*{1},
<11cm,0.7cm>*{\text{{\footnotesize $\tilde{g}-1$}}}, %x-axis label
<11.7cm,0.7cm>*{\text{{\footnotesize $\tilde{g}$}}}, %x-axis label
<0.8cm,4cm>*{\text{{\footnotesize $V_0$}}},        %y-axis label
<0.8cm,2cm>*{\text{{\footnotesize $2$}}},        %y-axis label
<0.8cm,1.5cm>*{\text{{\footnotesize $1$}}},        %y-axis label
<10.25cm,1.8cm>*{\overbrace{\text{\makebox[1.4cm]{}}}^{T_1=\rho_0}}, %overbraces
<9.45cm,2.3cm>*{\overbrace{\text{\makebox[3cm]{}}}^{T_2}}, %overbraces
<9cm,2.8cm>*{\overbrace{\text{\makebox[3.8cm]{}}}^{T_3}}, %overbraces
<7cm,3.85cm>*{\overbrace{\text{\makebox[7.7cm]{}}}^{T_{V_0-1}}}, %overbraces
<6cm,4.4cm>*{\overbrace{\text{\makebox[9.8cm]{}}}^{T_{V_0}=\tilde{g}}}, %overbraces
<8.7cm,1.7cm>*{\underbrace{\text{\makebox[1.5cm]{}}}_{T_2-T_1=\rho_1}},
\end{xy}
\caption{A graph to show the relationship between the $V_i$ and the $T_i$. We have also shown how $\rho_0$ and $\rho_1$ occur as the number of $V_i$ equal to one and two, respectively.}
\label{fig:tvdiagram}
 \end{figure}

For $m \geq 1$, it will be convenient to consider the quantities
\[T_m=|\{0\leq i <\tilde{g} \,|\, 0<V_i\leq m\}|.\]
These are illustrated in Figure~\ref{fig:tvdiagram}. We will show how to calculate these in terms of $\rho$. First we need to define the following collection of tuples for each $m\geq 0$:
\[S_m=\{\alpha \in \mathbb{Z}^{r+1}\colon \alpha_i \geq 0, 2m=\sum\alpha_i(\alpha_i + 1) \}.\]
\begin{lem}\label{lem:calcTi}
For $0\leq m<V_0$, we can calculate $T_m$ by
\[T_m = \max_{\alpha \in S_m} \rho \cdot \alpha\]
and $T_{V_0}$ satisfies
\[T_{V_0}=\tilde{g}=\frac{1}{2}\sum_{i=0}^t \rho_i^2-\rho_i\]
and
\[T_{V_0}\leq \max_{\alpha \in S_{V_0}} \rho \cdot \alpha.\]
\end{lem}
\begin{proof}
Since the $V_k$ form a decreasing sequence with $V_k=0$ if and only if $k\geq \tilde{g}$, we necessarily have $T_{V_0}=\tilde{g}$. Using Lemma~\ref{lem:modremoved}, we that $V_k=0$ if and only if there is $c\in \{\pm 1\}^{t+1}$ with $c\cdot \rho=2k-n$. The smallest of value $k$ for which this is true is $k=\frac{1}{2}(n-\sum_{i=0}^{t} \rho_i)$, which is obtained by taking $c=\{-1\}^{t+1}$. Thus we get $2\tilde{g}=\sum_{i=0}^{t} \rho_i^2-\rho_i$, as required (cf. \cite[Proposition 3.1]{greene2010space}).

\paragraph{}Now observe that for $0\leq m<V_0$, we have
\[T_m= \tilde{g}-\min\{k \colon V_k=m\}.\]
By Lemma~\ref{lem:modremoved}, $V_k = m$ and $0\leq k<n/2$ implies there is $c\in \Char(\mathbb{Z}^{t+1})$ such that $\norm{c}-t-1=8m$ and $c\cdot \rho = 2k-n$. If we write the coefficients of $c$ in the form $c_i=-(2\alpha_i +1)$, then $\sum_{i=0}^{t}(\alpha_i(\alpha_i+1)=2m$ and
\begin{equation}\label{eq:calcTiformula}
2k= n - \sum_{i=0}^t \rho_j - 2\alpha \cdot \rho=2\tilde{g} -2\alpha \cdot \rho.
\end{equation}
We see that for any $\alpha$ minimising \eqref{eq:calcTiformula}, we must have $\alpha \in S_m$, since it must satisfy $\alpha_i\geq 0$ for all $i$. Thus we we see that
\[T_m = \max_{\alpha \in S_m} \rho \cdot \alpha,\]
for $0\leq m<V_0$.
The equation \eqref{eq:calcTiformula} also shows that there must exist $\alpha$ satisfying $\sum_{i=0}^{t}\alpha_i(\alpha_i+1)=2V_0$ and $\alpha \cdot \rho = \tilde{g}$. This implies the inequality
\[T_{V_0}\leq \max_{\alpha \in S_{V_0}} \rho \cdot \alpha,\]
which completes the proof.
\end{proof}
\begin{rem}\label{rem:rho0rho1} It follows from this lemma that $T_1=\rho_0$ and $T_2=\rho_0+\rho_1$. In particular this implies that $\rho_1=T_2-T_1$. This is illustrated in Figure~\ref{fig:tvdiagram}.
\end{rem}
We now begin the process of showing how the remaining $\rho_i$ can be recovered from the sequence $(V_i)_{i\geq0}$. We begin with the simplest case, which is when $V_0\leq 1$.
\begin{lem}\label{lem:V0=1}
If $V_0\leq 1$, then $\tilde{g}\leq 3$ and $\rho$ takes the form
\[
\rho=
\begin{cases}
(1,1 , \dotsc, 1)   &\text{if }\tilde{g}=0\\
(2,1 , \dotsc, 1)   &\text{if }\tilde{g}=1\\
(2,2 ,1, \dotsc, 1) &\text{if }\tilde{g}=2\\
(3,1 , \dotsc, 1)   &\text{if }\tilde{g}=3.
\end{cases}\]
\end{lem}
\begin{proof}
If $V_0=0$, then $\tilde{g}=0$ and Lemma~\ref{lem:calcTi} implies that $\sum_{i=0}^t \rho_i^2-\rho_i=0$. This shows that we have $\rho_i=1$ for all $0\leq i \leq t$.

Suppose now that $V_0=1$. By Lemma~\ref{lem:calcTi}, we have
\[0<T_{1}=\tilde{g}=\frac{1}{2}\sum_{i=0}^t \rho_i^2-\rho_i\leq \max_{\alpha \in S_{1}} \rho \cdot \alpha.\]
Since $S_1$ consists of vectors with a single non-zero coordinate, which equals one, we have $\max_{\alpha \in S_{1}} \rho \cdot \alpha=\rho_0$. Thus we must have $\rho_0^2-\rho_0\leq 2\rho_0,$ and hence $\rho_0\leq 3$. If $\rho_0=3$, then we have
\[\tilde{g}=3+\frac{1}{2}\sum_{i=1}^t\rho_i(\rho_i-1)\leq \rho_0 =3,\]
which implies the $\rho_i=1$ for $1\leq i \leq t$ and $\tilde{g}=3$. If $\rho_0=2$, then $\tilde{g}\leq 2$ implies that $\rho_1\in \{1,2\}$, giving the other two possibilities in the statement of the lemma.
\end{proof}

From now on we will suppose that $V_0>1$. This allows us to define the quantity \[\mu = \min_{1\leq i<V_0} \{ T_i-T_{i-1}\}.\]
Since $T_1=\rho_0$ and $T_0=0$, we must have $\mu\leq \rho_0$.
\begin{lem}\label{lem:mubound}
If $\rho_0\geq 5$ or $\sum_{\rho_i \text{even}}\rho_i\geq 6$, then $\mu \leq 2$.
\end{lem}
\begin{proof}
For $m<V_0$, Lemma~\ref{lem:calcTi} shows that there is $\alpha \in S_m$ such that $\rho \cdot \alpha=T_m$. If $\alpha_l>0$, then consider $\alpha'$ defined by
\[\alpha'_i=
\begin{cases}
\alpha_i & i\ne l\\
\alpha_i-1 &i=l.
\end{cases}\]
By construction, we have $\alpha'\in S_{m-\alpha_l}$ and $\alpha'\cdot \rho = \alpha \cdot \rho - \rho_l=T_m-\rho_l.$
As $\alpha' \cdot \rho \leq T_{m-\alpha_l}$, we get
\begin{equation}\label{eq:mubound}
\rho_l\geq T_m-T_{m-\alpha_l}\geq \alpha_l \mu.
\end{equation}
If we have a maximiser $\alpha \in S_m$ such that $\rho \cdot \alpha=T_m$ and $\alpha$ does not satisfy
\begin{equation}\label{eq:rhobounds}
\alpha_i \leq
\begin{cases}
\frac{\rho_i-2}{2}    &\rho_i   \text{ even,}\\
\frac{\rho_i-3}{2}    &\rho_i>3 \text{ odd,}\\
\frac{\rho_i-1}{2}    & \rho_i\in\{1,3\},
\end{cases}
\end{equation}
for all $i$, then there is $l$ such that $\frac{\rho_l}{\alpha_l}<3$. So by \eqref{eq:mubound}, we see that $\mu\leq 2$. We will show that if $\rho$ satisfies the hypotheses of the lemma, then such a maximiser must exist.

\paragraph{}Let $c\in \Char(\mathbb{Z}^{t+1})$ be such that $c\cdot\rho=n$. By \eqref{eq:Vi1}, we have
\[8V_0\geq \norm{c} -t-1.\]
On the other hand, the Cauchy-Schwarz inequality implies that
\[|c\cdot\rho|^2=n^2\leq \norm{\rho}\norm{c}=n\norm{c},\]
showing that $\norm{c}\geq n$ with equality if and only if $c=\rho$. Altogether, this yields
\[V_0\geq \frac{\norm{\rho}-t-1}{8}=\frac{1}{8}\sum_{i=0}^t (\rho_i^2-1),\]
with equality if and only if $\rho\in \Char(\mathbb{Z}^{t+1})$.
We will let $N$ denote the quantity
\[N=\lfloor \frac{1}{8}\sum_{i=0}^t (\rho_i^2-1) \rfloor\leq V_0.\]

\paragraph{}Now take $\alpha \in S_m$, which satisfies the conditions given by \eqref{eq:rhobounds}.
It follows that
\begin{align}
\begin{split}\label{eq:maximiserbound}
m&=\frac{1}{2}\sum_{i=0}^t \alpha_i (\alpha_i +1)\\
 &\leq \sum_{\rho_i >3 \text{ odd}}\frac{(\rho_i-3)(\rho_i-1)}{8} + \sum_{\rho_i \text{ even}}\frac{\rho_i(\rho_i-2)}{8} + \sum_{\rho_i \in \{1,3\}} \frac{\rho_i^2-1}{8}\\
 &=\sum_{i=0}^t \frac{\rho_i^2-1}{8} + \sum_{\rho_i>3 \text{ odd}}\frac{(1-\rho_i)}{2} +
\sum_{\rho_i \text{ even}}\frac{1-2\rho_i}{8}.
\end{split}
\end{align}
\paragraph{} If $\rho_0$ is odd and $\rho_0\geq 5$, then \eqref{eq:maximiserbound} shows that
\[m\leq \sum_{i=0}^t \frac{\rho_i^2-1}{8} -2 < N-1\]
In particular, there is no $\beta \in S_{N-1}$ satisfying \eqref{eq:rhobounds}. Since $N-1< V_0$, there is $\beta \in S_{N-1}$ with $\beta \cdot \rho=T_{N-1}$ and so \eqref{eq:mubound} implies that $\mu\leq 2$.
\paragraph{}If $\sum_{\rho_i \text{ even}}\rho_i\geq 6$, then we must have $\sum_{\rho_i \text{ even}}(2\rho_i-1)\geq \frac{3}{2}\sum_{\rho_i \text{ even}}\rho_i \geq 9$. Therefore, \eqref{eq:maximiserbound} shows that
\[m< \sum_{i=0}^t \frac{\rho_i^2-1}{8} -1 < N.\]
In particular, there is no $\beta \in S_{N}$ satisfying \eqref{eq:mubound}. Since we are assuming there is an even $\rho_i$, we have $N<V_0$ and so there exists $\beta\in S_{N}$ such that $\beta \cdot \rho=T_{N}$ and so \eqref{eq:mubound} implies that $\mu\leq 2$.
\end{proof}
If $\mu>2$, then $\rho$ must fall into one of a small number of cases.
\begin{lem}\label{lem:mu=3}
If $\mu>2$ then either $T_1=3$ or $T_1=4$. If $T_1=3$, then $\rho$ takes the form
\[
\rho=
\begin{cases}
(\underbrace{3, \dotsc , 3}_{d}, 1, \dotsc, 1)      &\text{if }\tilde{g}=3d\\
(\underbrace{3, \dotsc , 3}_{d},2, 1, \dotsc, 1)    &\text{if }\tilde{g}=3d+1\\
(\underbrace{3, \dotsc , 3}_{d},2,2, 1, \dotsc, 1)  &\text{if }\tilde{g}=3d+2.
\end{cases}
\]
If $T_1=4$, then $\rho$ must take the form
\[\rho=(4,\underbrace{3, \dotsc , 3}_{d}, 1, \dotsc, 1)   \quad\text{where }\tilde{g}=3d+6\]
\end{lem}
\begin{proof}
If $\mu>2$, then Lemma~\ref{lem:mubound} and the observation that $\mu \leq T_1=\rho_0$, we must have $\rho_0 \in \{3,4\}$. If $\rho_0=3$, then Lemma~\ref{lem:mubound} implies that we have $\rho_i=2$ for at most two values $i$. If $\rho_0=4$, then Lemma~\ref{lem:mubound} implies that $\rho_i$ is odd for all $i\geq 1$. It is then easy to deduce that $\rho$ must take the required form by using the formula $\tilde{g}=\frac{1}{2}\sum_{i=0}^t \rho_i^2-\rho_i$.
\end{proof}
\begin{rem}
Although it suffices for our purposes, Lemma~\ref{lem:mu=3} does not quite tell the full story. If $\rho=(4,3, \dotsc , 3, 1, \dotsc, 1)$, then one can show that we have $\mu=1$. This shows that the only cases with $\mu>2$ are those given in Lemma~\ref{lem:mu=3} with $\rho_0=3$. For these examples we do have $\mu=3$.
\end{rem}
Now we show that the sequence $(V_i)_{i\geq0}$ determines $\rho$ when $\mu\leq 2$.
\begin{lem}\label{lem:mu<3}
If $\mu \leq 2$, then the vector $\rho$ satisfying \eqref{eq:Vi1} is unique.
\end{lem}
\begin{proof}
We will show that can calculate the coefficients of $\rho$ iteratively from the values $T_0<T_1< \dotsb <T_{V_0}=\tilde{g}$.
Using the $T_i$, we will construct a sequence $s^{(0)}, s^{(1)}, \dots, s^{(N)}$, which we will show to satisfy
\[s^{(k)}=(\rho_0,\dotsc, \rho_k,0,\dots,0),\]
for each $k\leq N$.
The integer $N$ will be large enough that $S^{(N)}$ satisfies
\[\max_{\alpha \in S_t} s^{(N)}\cdot \alpha = T_t\]
for all $t<V_0$. We will show we can deduce $\rho_i$ for any $i>N$ by considering $T_{V_0}=\tilde{g}$.
\paragraph{}Start by setting
\[s^{(0)}=(T_1,0, \dotsc, 0)=(\rho_0,0, \dotsc, 0).\]
Now suppose that for $l\geq 0$ we have $s^{(l)}_i=\rho_i$ for all $i\leq l$. Suppose there is $t<V_0-1$ minimal such that $M=\max_{\alpha \in S_t} s^{(l)}\cdot \alpha < T_t$.
\paragraph{Claim 1} We have $\rho_{l+1}=T_{t}-T_{t-1}$.
\begin{proof}[Proof of Claim 1.]
Let $\alpha\in S_{t-1}$ be such that $s^{(l)}\cdot \alpha=T_{t-1}$, Such an $\alpha$ must also satisfy $\rho \cdot \alpha=T_{t-1}$. In particular $\alpha_i=0$ for $i>l$.

Now we consider $\alpha'\in S_t$ defined by
\[\alpha'_i=
\begin{cases}
\alpha_i & i\ne l+1\\
1 &i=l+1.
\end{cases}\]
We have $\alpha'\cdot \rho = T_{t-1}+\rho_{l+1}\leq T_t$. This implies that
\begin{equation}\label{eq:algproof1}
\rho_{l+1}\leq T_t-T_{t-1}.
\end{equation}

Let $\beta \in S_{t}$, be such that $\rho \cdot \beta = T_t$. Since $M<T_t$, we may assume $\beta_{l+1}>0$. Thus we can define $\beta'$ by
\[\beta'_i=
\begin{cases}
\beta_i & i\ne l\\
\beta_i-1 &i=l+1.
\end{cases}\]
We have $\beta'\in S_{t-\beta_l}$. Therefore we obtain
\begin{equation}\label{eq:algproof2}
T_{t-1}\geq T_{t-\beta_l} \geq \rho \cdot \beta' =  T_{t}-\rho_{l+1}.
\end{equation}
Combining \eqref{eq:algproof1} and \eqref{eq:algproof2} that $\rho_{l+1}=T_{t}-T_{t-1}$, as claimed.
\end{proof}
Thus if we define $s^{(l+1)}$ by
\[s^{(l+1)}_i=
\begin{cases}
s^{(l)}_i & i\ne l+1\\
T_{t}-T_{t-1} &i=l+1,
\end{cases}\]
we see that $s^{(l+1)}$ satisfies
\[s^{(l+1)}=(\rho_0,\dotsc, \rho_{l+1},0,\dots,0)\]
and
\[s^{(l+1)}\cdot \alpha'=\max_{\alpha \in S_t} \alpha \cdot s^{(l+1)}=T_t,\]
where $\alpha'\in S_t$ is as defined in the proof of Claim~1.

\paragraph{}Proceeding in this way, we eventually obtain $s^{(N)}$, such that $T_t=\max_{\alpha \in S_t} \alpha \cdot s^{(N)}$, for all $0\leq t <V_0$ and
\[s^{(N)}=(\rho_0,\dotsc, \rho_{N},0,\dots,0).\]
\paragraph{Claim 2} We have $\rho_l\leq \mu\leq 2$, for all $l>N$.
 \begin{proof}[Proof of Claim 2.]Let $\tau<V_0-1$ be such that $T_{\tau+1}-T_{\tau}=\mu$. There is $\alpha \in S_{\tau}$ such that $\alpha \cdot \rho = \alpha \cdot s^{(N)} = T_{\tau}$. Such an $\alpha$ must satisfy $\alpha_l=0$ for $l>N$. Let $\alpha'\in S_{\tau +1}$ be defined by
\[\alpha'_i=
\begin{cases}
\alpha_i & i\ne l\\
1 &i=l.
\end{cases}\]
We have
\[\rho_l=\alpha'\cdot \rho-T_{\tau}\leq T_{\tau+1}-T_{\tau}=\mu\leq 2,\]
as required.
\end{proof}
It remains to determine how many values of $i>N$ satisfy $\rho_i=2$. Since we have the formula $T_{V_0}=\frac{1}{2}\sum_{i=0}^t \rho_i(\rho_i-1)$, we see that there are
\[T_{V_0}-\frac{1}{2}\sum_{i=0}^t s^{(N)}_i(s^{(N)}_i-1)\]
values of $i>N$ with $\rho_i=2$. Since $\rho_i=1$ for all remaining values of $i$, this shows that $\rho$ is determined by the $T_i$.
\end{proof}

The proof of Lemma~\ref{lem:mu<3} combined with Lemma~\ref{lem:V0=1} and Lemma~\ref{lem:mu=3} provides an algorithm for calculating $\rho$. This shows that $\rho$ is the unique vector with $\rho_0\geq \rho_1 \geq \dotsb \geq \rho_t > 0$ and $\norm{\rho}=n$ satisfying \eqref{eq:Vi1}. Moreover, if we take $m$ to be maximal such that $\rho_m>1$, then this algorithm calculates the tuple $(\rho_0, \dots, \rho_m)$ using only the sequence $(V_i)_{i\geq 0}$.

This allows us to deduce Theorem~\ref{thm:Gibbonssouped} and Corollary~\ref{cor:interfromunique} from Theorem~\ref{thm:Gibbons}.
\begin{proof}[Proof of Theorem~\ref{thm:Gibbonssouped}]
Theorem~\ref{thm:Gibbons} shows that the intersection form $Q_X$ takes the form of a $p/q$-changemaker lattice,
\[-Q_X\cong L=\langle w_0, \dotsc, w_l\rangle^\bot \subseteq \mathbb{Z}^{t+s+1},\]
where the sequence $(V_i)_{i\geq 0}$, which is an invariant of $K$, can be calculated from $w_0= \sigma_t f_t+ \dotsb + \rho_1 f_1 +e_0$ by the formula \eqref{eq:CMformula}. Thus, $w_0$ satisfies \eqref{eq:Vi1} and using the algorithm provided by the proof of Lemma~\ref{lem:mu<3}, Lemma~\ref{lem:V0=1} and Lemma~\ref{lem:mu=3}, we see the tuple $(\sigma_m, \dots , \sigma_t)$, where $m$ is minimal such that $\sigma_m >1$, is independent of $t$ and $\norm{w_0}=\lceil \frac{p}{q} \rceil$. By definition, $(\sigma_m, \dots , \sigma_t)$ are the stable coefficients of $L$ and it follows that they are independent of $b_2(X)$ and $p/q$.
\end{proof}
\begin{proof}[Proof of Corollary~\ref{cor:interfromunique}]
This follows combining Theorem~\ref{thm:Gibbonssouped} with Remark~\ref{rem:CMdetermined}. Theorem~\ref{thm:Gibbonssouped} shows that $-Q_X$ and $-Q_{X'}$ are both $p/q$-changemaker lattices with the same stable coefficients. Remark~\ref{rem:CMdetermined} then shows that $Q_{X'} \cong Q_{X} \oplus (-\mathbb{Z}^{k})$. The isomorphism of intersection forms $Q_{X} \oplus (-\mathbb{Z}^{k})\cong Q_{X\#_{k}\overline{\mathbb{CP}}^2}$ is clear.
\end{proof}

\subsection{$L$-space knots}\label{sec:Lspacesurgery}
Now we specialise to the case of $L$-space surgeries. A knot $K$ is said to be an {\em $L$-space knot} if $S^3_{p/q}(K)$ is an $L$-space for some $p/q \in \mathbb{Q}$. The knot Floer homology of an $L$-space knot is known to be determined by its Alexander polynomial, which can be written in the form
\[\Delta_K(t)=a_0 \sum_{i=1}^g a_i(t^i+t^{-i}),\]
where $g=g(K)$, $a_g=1$ and the non-zero values of $a_i$ alternate in sign and assume values in $\{\pm 1\}$ \cite{Ozsvath04genusbounds},\cite{Ozsvath05Lensspace}. Given an Alexander polynomial in this form, we can compute its {\em torsion coefficients} by the formula
\[t_i(K) = \sum_{j\geq 1}ja_{|i|+j}.\]
When $K$ is an $L$-space knot, the $V_i$ appearing in \eqref{eqn:NiWuVonly} satisfy $V_i=t_i(K)$ for $i\geq 0$ \cite{ozsvath2011rationalsurgery}. Thus if $S^3_{p/q}(K)$ is an $L$-space bounding a negative-definite sharp $4$-manifold $X$, then Theorem~\ref{thm:Gibbonssouped} shows that the intersection form is isomorphic to a $p/q$-changemaker lattice $L$, where the stable coefficients, $(\sigma_r, \dotsc, \sigma_m)$, are determined by the torsion coefficients. Since $t_i(K)=0$ if and only if $i\geq g(K)$, Lemma~\ref{lem:calcTi} shows that the genus can be computed by the formula
\begin{equation}\label{eq:calculateg}
g(K)=\frac{1}{2}\sum_{i=m}^{r} \sigma_i(\sigma_i-1),
\end{equation}
which was first proven by Greene \cite[Proposition~3.1]{greene2010space}.
\begin{rem}\label{rem:LspaceTi}
Lemma~\ref{lem:calcTi} shows that $\sigma_{r}$ and $\sigma_{r-1}$ have particularly simple interpretations in terms of torsion coefficients:
\[\sigma_r = \# \{0\leq i< g | t_i(K) =1\} \text{ and } \sigma_{r-1} = \# \{0\leq i< g | t_i(K) =2\}.\]
As in the proof of Lemma~\ref{lem:mu<3}, the remaining stable coefficients can be also be computed from the torsion coefficients. However, the relationship is more complicated.
\end{rem}

\section{Graph lattices and obtuse superbases}\label{sec:graphlattices}
In this section, we gather together some lattice-theoretic concepts and properties that we will need.
\subsection{Graph lattices}
We recall the definition of a graph lattice and state the results that we will require for this paper. All statements in this section can be found with proof in \cite{mccoy2013alternating}.

\paragraph{} Let $G=(V,E)$ be a finite, connected, undirected graph with no self-loops. For a pair of disjoint subsets $R,S \subset V$, let $E(R,S)$ be the set of edges between $R$ and $S$. Define $e(R,S)=|E(R,S)|$. We will use the notation $d(R)=e(R,V\setminus R)$.

\paragraph{} Let $\overline{\Lambda}(G)$ be the free abelian group generated by $v\in V$. Define a symmetric bilinear form on $\overline{\Lambda}(G)$ by
\[
v\cdot w =
  \begin{cases}
   d(v)            & \text{if } v=w \\
   -e(v,w)       & \text{if } v\ne w.
  \end{cases}
\]
In this section we will use the notation $[R]=\sum_{v\in R}v$, for $R\subseteq V$. The above definition gives
\begin{equation}\label{eq:vdotsubgraph}
v\cdot [R] =
  \begin{cases}
   -e(v,R)            & \text{if } v\notin R \\
   e(v,V\setminus R)       & \text{if } v\in R.
  \end{cases}
\end{equation}

From this it follows that $[V]\cdot x= 0$ for all $x \in \overline{\Lambda}(G)$. We define the {\em graph lattice} of $G$ to be
\[\Lambda(G):= \frac{\overline{\Lambda}(G)}{\mathbb{Z}[V]}.\]
The bilinear form on $\overline{\Lambda}(G)$ descends to $\Lambda(G)$. Since we have assumed that $G$ is connected, the pairing on $\Lambda(G)$ is positive-definite. This makes $\Lambda(G)$ into an integral lattice. Henceforth, we will abuse notation by using $v$ to denote its image in $\Lambda(G)$.

%The following lemma will be useful.
%\begin{lem}\label{lem:usefulbound}
%Let $x=[R]$ be a sum of vertices, then for any $z\in \Lambda(G)$, we have
%\[(x-z)\cdot z\leq 0.\]
%\end{lem}

Recall that a vector $z$ in a lattice is {\em irreducible} if it cannot be written in the form $z=x+y$ for non-zero $x$ and $y$ with $x\cdot y \geq 0$. The irreducible vectors in $\Lambda(G)$ can be characterised in terms of the graph $G$.

\begin{lem}\label{lem:irreducible}
The vector $x \in \Lambda(G)\setminus \{0\}$ is irreducible if and only if $x=[R]$ for some $R\subseteq V$ such that $R$ and $V\setminus R$ induce connected subgraphs of $G$.\qed
\end{lem}

A connected graph is said to be {\em 2-connected} if it cannot be disconnected by deleting a vertex. This property is equivalent to $\Lambda(G)$ being {\em indecomposable}, that is, $\Lambda(G)$ cannot be written as the orthogonal direct sum $\Lambda(G)= L_1 \oplus L_2$ with $L_1,L_2$ non-zero sublattices.
\begin{lem}\label{lem:2connectgraphlat}
The following are equivalent:
\begin{enumerate}[(i)]
\item The graph $G$ is 2-connected;
\item Every vertex $v\in V$ is irreducible;
\item The lattice $\Lambda(G)$ is indecomposable.
\end{enumerate}\qed
\end{lem}
Given a graph lattice of some graph $G$, the following lemma will be useful for identifying other graphs  with isomorphic graph lattices.
\begin{lem}\label{lem:cutedge}
Suppose that $G$ is 2-connected. Let $v$ be a vertex such that we can find $x,y\in \Lambda(G)$, with $v=x+y$ and $x\cdot y=-1$. Then there is a cut edge $e$ in $G\setminus \{v\}$ and if $R,S$ are the vertices of the two components of $(G\setminus \{v\})\setminus\{e\}$ then $\{x,y\}=\{[R]+v,[S]+v\}$. Let $u_1$ and $u_2$ be the endpoints of $e$. These are the unique vertices $u_1,u_2\ne v$, with $x\cdot u_1=y\cdot u_2=1$. Furthermore, any vertex $w \notin \{v,u_1,u_2\}$ satisfies $w\cdot x,w\cdot y\leq 0$. \qed
\end{lem}

\subsection{Obtuse superbases}\label{sec:vertexsuperbases}
Given a positive definite integral lattice $L$ of rank $r$, we say that $L$ {\em admits an obtuse superbase} if it contains a set $B=\{v_0, \dotsc, v_r\}$, such that $v_1,\dotsc, v_r$ form a basis for $L$, $v_0+ \dotsb + v_r=0$ and $v_i\cdot v_j \leq 0$ for all $0\leq i\ne j\leq r$. We will call the set $B$ a {\em an obtuse superbase} for $L$. This terminology is taken from the work of Conway and Sloane \cite{Conway92lowdimlattices6}.

Given an obtuse superbase $B=\{v_0, \dotsc, v_r\}$ for $L$, we can construct a graph $G_B$ by taking vertex set $B$ with $|v_i\cdot v_j|$ edges between vertices $v_i$ and $v_j$ for $i\ne j$. With this construction in mind, we will frequently refer to elements of a given obtuse superbase as vertices of $L$.

\begin{prop}\label{prop:graphsuperbase}
The graph $G_B$ is connected and $L$ is isomorphic to $\Lambda(G_B)$.
\end{prop}
\begin{proof}
First we show that $G_B$ is connected. Let $R\subseteq B$ be the vertices of a non-empty connected component of $G_B$. We see that the vector $[R]=\sum_{x\in R}x$ satisfies $[R]\cdot v_i=0$ for all $0\leq i\leq r$ (cf. \eqref{eq:vdotsubgraph}). Since $L$ is positive-definite, this implies that $[R]=0$. By definition, $v_1,\dotsc, v_r$ must be linearly independent. It follows that $R=B$ and hence $G_B$ is connected, as required.

To show that $\Lambda(G_B)$ is isomorphic to $L$, take the linear map which takes vertices to the corresponding vectors in $L$. Since $v_0+ \dotsb + v_r=0$, we have
\[d(v_k)=-\sum_{i\ne k}v_k\cdot v_i=\norm{v_k},\]
and by construction we have $e(v_i,v_j)= - v_i\cdot v_j$, for $i\ne j$. This shows that this map is the required isomorphism.
\end{proof}

For any given lattice there may be many choices of obtuse superbase. The following lemma shows one way to convert one obtuse superbase into another.
\begin{lem}\label{lem:superbasemod}
Let $L$ be an indecomposable lattice with an obtuse superbase $B$. Suppose that we have $v \in B$ which can be written as $v=x+y$, where $x,y\in L$ and $x\cdot y=-1$. There are unique $u_1,u_2 \in B$ with $u_1\cdot x>0$ and $u_2\cdot y>0$ and the set $B'=(B\setminus\{v,u_1,u_2\})\cup \{x,y,u_1+u_2\}$ is also an obtuse superbase for $L$.
\end{lem}
\begin{proof}
Since $L$ is indecomposable, Lemma~\ref{lem:2connectgraphlat} shows that the graph $G_B$ is 2-connected. Thus we may apply Lemma~\ref{lem:cutedge}, which shows that there are disjoint connected subgraphs $G_1$ and $G_2$ of $G_B$ and vertices $u_1$ and $u_2$, such that $x=v+u_1+ \sum_{z\in G_1}z$ and $y=v+u_2+ \sum_{z\in G_2}$, with a unique edge between $u_1$ and $u_2$ which is a cut-edge in $G_B \setminus \{v\}$. It is straight-forward to verify that $B'=(B\setminus\{v,u_1,u_2\})\cup \{x,y,u_1+u_2\}$ is an obtuse superbase for $L$. The an illustration of how the graph $G_{B'}$ is obtained from $G_B$ is given in Figure~\ref{fig:sbtransform}.
\begin{figure}
  \centering
  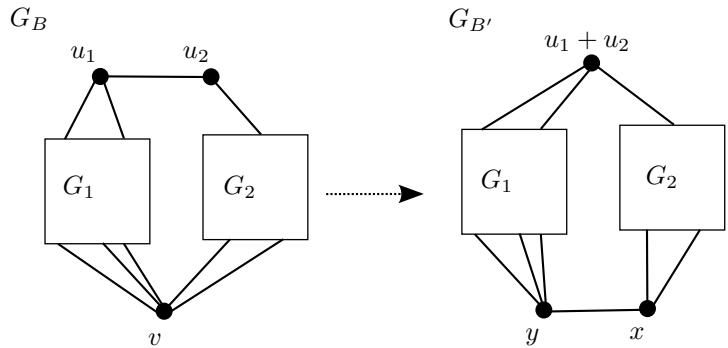
 \caption{The graphs $G_B$ and $G_{B'}$ corresponding to the obtuse superbases appearing in Lemma~\ref{lem:superbasemod}.}
 \label{fig:sbtransform}
\end{figure}
\end{proof}

\section{Alternating surgeries}\label{sec:mainresults}
In this section, we will prove our main results.
\subsection{The Goeritz form}
A diagram $D$ of a link $L$ divides the plane into connected regions. We may colour these regions black and white in a chessboard fashion. This colouring can be done in two different ways. Each of the possible colourings gives an incidence number, $\mu(c)\in \{\pm 1\}$, at each crossing $c$ of $D$, as shown in Figure~\ref{fig:incidencenumber}.
 \begin{figure}[h]
  \centering
  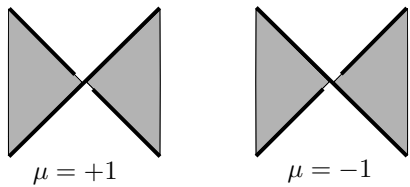
 \caption{The incidence number of a crossing.}
 \label{fig:incidencenumber}
\end{figure}
 We construct a planar graph, $\Gamma_D$, by drawing a vertex in each white region and an edge $e$ for every crossing $c$ between the two white regions it joins. We define an incidence number on each edge  by $\mu(e)=\mu(c)$. We call this the {\em white graph} corresponding to $D$. This gives rise to a {\em Goeritz matrix}, $G_D=(G_{ij})$, defined by labeling the vertices of $\Gamma_D$, by $v_1,\dotsc , v_{r+1}$ and for $1\leq i,j \leq r$, setting
 \[g_{ij}=\sum_{e \in E(v_i,v_j)}\mu(e)\]
 for $i\ne j$ and
 \[g_{ii} = - \sum_{e \in E(v_i, \Gamma_D\setminus v_i)}\mu(e)\]
 otherwise \cite[Chapter 9]{lickorish1997introduction}.

\paragraph{} Now suppose that $L$ is an alternating, non-split link. If $D$ is any alternating diagram, then we may fix the colouring so that $\mu(c)=-1$ for all crossings. In this case, $G_D$ defines a positive-definite bilinear form. This in turn gives a lattice, $\Lambda_D$ which we will refer to as the {\em white lattice} of $D$. Observe that if $D$ is reduced (i.e. contains no nugatory crossings), then $\Gamma_D$ contains no self-loops or cut-edges and $\Lambda_D$ is isomorphic to the graph lattice $\Lambda(\Gamma_D)$.

\paragraph{}Ozsv{\'a}th and Szab{\'o} have shown that the Heegaard Floer homology $d$-invariants of the branched double cover $\Sigma(L)$ are determined by $\Lambda_D$ \cite{ozsvath2005heegaard}.
\begin{thm}[\cite{ozsvath2005heegaard}]\label{thm:altboundssharp}
Let $L$ be a non-split alternating link with a reduced alternating diagram $D$. The double branched cover $\Sigma(L)$ is an $L$-space which bounds a simply-connected negative-definite sharp 4-manifold with intersection form isomorphic to $-\Lambda_D$.\qed
\end{thm}

\subsection{Changemaker lattices admitting obtuse superbases}
We will establish some restrictions on changemaker lattice which admits an obtuse superbase. The following proposition, which combines results from \cite{mccoy2013alternating} and \cite{mccoy2014noninteger}, will allow us to restrict our attention to integer changemaker lattices.
\begin{prop}\label{prop:nearby}
Suppose that for some $p/q=n-r/q$ with $q>r\geq 1$, the changemaker lattice
\[L_{p/q}=\langle w_0, \dotsc, w_l\rangle^\bot \subseteq \mathbb{Z}^{t+s+1},\]
where $w_0=e_0+ \sigma_1 f_1 + \dotsb + \sigma_t f_t$
admits an obtuse superbase. Then the changemaker lattices
\[
L_{n}=\langle w_0\rangle^\bot \subseteq \mathbb{Z}^{t+1}
=\langle e_0, f_1, \dotsc, f_t \rangle
\]
and
\[
L_{n-1}=\langle w_0-e_0\rangle^\bot \subseteq \mathbb{Z}^{t}
=\langle f_1, \dotsc, f_t \rangle,
\]
both admit obtuse superbases.
Furthermore, if $\sigma_t>1$, then we can assume the obtuse superbase for $L_{n-1}$ contains a vector $x$ with $x\cdot f_1=-2$.
\end{prop}
\begin{proof}
Since $L_{p/q}$ is admits an obtuse superbase, it follows from \cite[Proposition~7.7]{mccoy2014noninteger} that the lattice
\[
L_{n-\frac{1}{2}}=\langle w_0, e_1-e_0\rangle^\bot \subseteq \mathbb{Z}^{t+2}
=\langle e_1,e_0, f_1, \dotsc, f_t \rangle
\]
also admits an obtuse superbase, which we will call $B$. The results of \cite{mccoy2013alternating} show that there are precisely two vertices $v$ and $w$ in $B$ with $v\cdot e_0,w\cdot e_0\ne 0$ and they satisfy $v\cdot w\leq -1$. Moreover, the results of the same paper show that we can assume that $v=-f_1+e_0+e_1$ and $w\cdot e_0=w\cdot e_1=-1$ and if there is $k$ such that $\sigma_k>1$, then we can assume that $w\cdot f_1=-1$.

\paragraph{}Consider the set $B'=B\setminus\{v,w\}\cup \{v+w\}$. Since $(v+w)\cdot e_0=(v+w)\cdot e_1=0$, we have $B'\subset L_{n-1}$. Since $B$ spans $L_{n-\frac{1}{2}}$, we see that $B'$ must span $L_{n-1}$. Since $B$ is an obtuse superbase for $L_{n-\frac{1}{2}}$, it follows that $B'$ is an obtuse superbase for $L_{n-1}$, where the graph $G_{B'}$ is obtained from $G_B$ by contracting the edge between $v$ and $w$. Furthermore, if there is $\sigma_k>1$, then $x=v+w$ is the required vector with $x\cdot f_1 = -2$.

\paragraph{}Now consider the set $B''=B\setminus\{v,w\}\cup \{v-e_1,w+e_1\}$. Since every element $x\in B\setminus\{v,w\}$ has $x\cdot e_1=0$, we see that every $x\in B''$ satisfies $x\cdot e_1=0$, so we have $B''\subseteq \langle e_0, f_1, \dotsc , f_t \rangle$ and hence $B''\subseteq L_{n+1}$. Since $B$ is an obtuse superbase for $L_{n-\frac{1}{2}}$ and $v\cdot w \leq -1$, it follows that $B''$ is an obtuse superbase for $L_{n+1}$, where the graph $G_{B''}$ is obtained from $G_B$ by deleting an edge between $v$ and $w$.
\end{proof}

The next lemma gives bounds on when a changemaker lattice can be decomposable.
\begin{lem}[\cite{GreeneLRP}, Lemma~5.1]\label{lem:reduciblebound}
Suppose that $L=\langle w_0 \rangle^\bot\subseteq \mathbb{Z}^{t}$ is a changemaker lattice,  where $w_0=\sigma_1 f_1 + \dotsb + \sigma_t f_t$ with $\sigma_i\geq 1$ for all $i$ and $\sigma_t>1$. Let $m\leq t$ be minimal such that $\sigma_m>1$. If $L$ is decomposable, then $\sigma_m=m-1$.\qed
\end{lem}
%\begin{proof}[Proof (sketch).]
%For each $2\leq k\leq t$, the changemaker condition combined with Proposition~\ref{prop:CMprop} implies that either $\sigma_k=1+\sigma_1+ \dotsb + \sigma_{k-1}$ or there is $A_k \subseteq \{1, \dotsc, k-2\}$ such that $\sigma_k -\sigma_{k-1} = \sum_{i \in A_k} \sigma_i$. Thus we can define $v_k\in L$ by $v_k=-f_k + f_{k-1}+ \dotsb + f_2 + 2f_1$ or $v_k=-f_k+f_{k-1}+\sum_{i \in A_k}f_i$. Each $v_k$ constructed in this way is irreducible (cf. \cite[Section~3.4]{GreeneLRP}). Taking a set $\{v_2, \dotsc, v_t\}$ constructed in this fashion gives a basis for $L$.
%
%We will show that $L$ is indecomposable if $\sigma_m\ne m-1$. Suppose that $L$ can be written as an orthogonal direct sum $L=L_1 \oplus L_2$. Since each $v_k$ is irreducible, we must have $v_k \in L_1$ or $v_k \in L_2$. Without loss of generality, we may assume that $v_2\in L_1$. For each $2<k<m$, we have $v_k=e_{k-1}-e_k$ and $v_k\cdot v_{k-1}=-1$ so it follows that $\{v_2,\dotsc, v_{m-1}\}\subseteq L_1$. If $\sigma_m\ne m-1$, then $v_m\ne -e_m + e_{m-1}+\dotsb + e_1$, which implies we can find $k<m$ such that $v_m\cdot v_l\ne 0$. Therefore $v_m$ is also in $L_1$. For any $k>m$ we can find $l<k$ such that $v_k\cdot v_l \ne 0$. If $v_k\cdot e_1=0$, then let $l$ be minimal such that $v_k\cdot e_l\ne 0$, giving $v_k\cdot v_l=-1$. If $v_k\cdot e_1\ne0$, then we have at least one of $v_2\cdot v_k$ or $v_2\cdot v_m$ non-zero. This shows that we must have $v_k\in L_1$ for all $k$. Since the $v_k$ form a basis, this shows that $L=L_1$. This proves that $L$ is indecomposable.
%\end{proof}
We get a similar bound on a changemaker lattice admitting an obtuse superbase in terms of its stable coefficients. This will allow us to prove the upper bound in Theorem~\ref{thm:widthbound}.
\begin{lem}\label{lem:graphbound}
Suppose that $L=\langle w_0 \rangle^\bot\subseteq \mathbb{Z}^{t}$, is a changemaker lattice, where $w_0=\sigma_1 f_1 + \dotsb + \sigma_t f_t$ with $\sigma_i\geq 1$ for all $i$. If the stable coefficients $(\sigma_m, \dots, \sigma_t)$ are a non-empty tuple and $L$ admits an obtuse superbase, then $\sigma_{m}\geq m-2$ and
\[\norm{w_0}\leq 1 + \sigma_m + \sum_{i=m}^t \sigma_i^2.\]
\end{lem}
\begin{proof}
If $L$ is decomposable, then Lemma~\ref{lem:reduciblebound} shows that the bound is automatically satisfied. We will assume from now on that $L$ is indecomposable.

For $2\leq i \leq m-1$, let $v_i$ be the vector $v_i=e_i-e_{i-1}$. Since $\sigma_i=\sigma_{i-1}=1$ for $i$ in this range, we have $v_i\in L$. We will use Lemma~\ref{lem:superbasemod} to show that $L$ admits an obtuse superbase containing the vectors $v_2, \dotsc, v_{m-1}$.

Let $B$ be an obtuse superbase and let $k\leq m-1$, be minimal such that $v_k$ is not in $B$. Suppose first that $k=2$. Since $v_2$ is irreducible, Lemma~\ref{lem:irreducible} implies that it can be written as a sum of elements of $B$. Hence, there is a vector $u\in B$ with $u\cdot v_2>0$. By Lemma~\ref{lem:2connectgraphlat}, the indecomposability of $L$ implies that $u$ is irreducible. In turn, this implies that $(u-v_2)\cdot v_2=u\cdot v_2 -2 = -1$. Therefore by applying Lemma~\ref{lem:superbasemod}, we see that that there is an obtuse superbase containing $v_2$.

Now we suppose that $k>2$. Since $v_k$ is irreducible, Lemma~\ref{lem:irreducible} shows that it can be written as a sum of elements of $B$. Since $v_{k-1}$ is a vertex of $B$ and $v_k\cdot v_{k-1}=-1$, there is $u\in B$ with $u\cdot v_k = -u\cdot v_{k-1}=1$. This must satisfy $(u-v_k)\cdot v_k= -1$. By Lemma~\ref{lem:superbasemod}, this implies we can find an obtuse superbase containing $v_k$. Moreover, since $(u-v_k)\cdot v_j \leq 0$ and $v_k\cdot v_j \leq 0$ for all $2\leq j<k$, so can assume that $v_2, \dots, v_{k-1}$ are also in this obtuse superbase. Thus proceeding inductively, we see that we can assume that $v_2, \dots, v_{m-1}$ are all contained in the obtuse suberbase $B$.

\paragraph{} Suppose that $\sigma_m=m-b$ for some $m-2\geq b\geq 2$. Consider the vector
$v_m=-e_m+e_{m-1}+ \dotsb + e_{b}\in L$. Since this is irreducible, Lemma~\ref{lem:irreducible} shows that we may write it as a sum of vertices $v_m = \sum_{x\in R} x$ for some subset $R\subseteq B$. Since $v_m\cdot v_b = -1$, we have $v_b \notin R$ and there must exist $u\in R$ with $u\cdot v_b =-1$ and $u\cdot v_m=1$. However, as $\norm{v_b}=2$ there are at most two vectors in $B$ which pair nontrivially with $v_b$. If $b\geq 3$ then, we have $v_{b-1}\cdot v_b = v_{b+1}\cdot v_b =-1$ and $v_{b-1}\cdot v_m =v_{b+1}\cdot v_m =0$. This implies that the required $u\in B$ cannot exist if $b\geq 3$. Thus we must have $b=2$. This shows that $\sigma_m\geq m-2$, as required. Since $\sigma_i=1$ for $i<m$, we have
\[\norm{w_0}=m-1 + \sum_{i=m}^t \sigma_i^2\leq 1+\sigma_m+\sum_{i=m}^t \sigma_i^2,\]
which is the required bound. This completes the proof.
\end{proof}

This allows us to prove the inequality which will give Theorem~\ref{thm:upperbound}.
\begin{lem}\label{lem:gbound}
Suppose that $L=\langle \sigma_1 f_1 + \dotsb + \sigma_t f_t \rangle^\bot\subseteq \mathbb{Z}^{t}$ is a changemaker lattice which admits an obtuse superbase and $\sigma_t>1$. Then
\[\sum_{i=1}^t \sigma_i^2 \leq 2\sum_{i=1}^t \sigma_i (\sigma_i -1) + 3\]
\end{lem}
\begin{proof}
Let $m$ be minimal such that $\sigma_m>1$. Since $L$ admits an obtuse superbase, Lemma~\ref{lem:graphbound} shows that we have
\[\sum_{i=1}^t \sigma_i^2\leq \sum_{i=m}^t \sigma_i^2 + \sigma_m +1.\]
Observe that if $\sigma_i \geq 2 $, then $\sigma_i^2 \leq 2 \sigma_i (\sigma_i -1)$. Since $\sigma_m\geq 2$, we also have $\sigma_m^2 + \sigma_m \leq 2\sigma_m (\sigma_m -1) +2$. Combining these inequalities, we obtain
\begin{align*}
\sum_{i=1}^t \sigma_i^2 &\leq \sum_{i=m}^t \sigma_i^2 + \sigma_m +1 \\
                        &\leq 2\sum_{i=m}^t \sigma_i(\sigma_i-1) +3 \\
                        &=2\sum_{i=1}^t \sigma_i (\sigma_i -1) + 3,
\end{align*}
which is the required inequality.
\end{proof}

\subsection{The main results}\label{subsec:mainresults}
Suppose that $K$ is an nontrivial knot such that $S^3_{p/q}(K)$ is an alternating surgery, that is $S_{p/q}^3(K)=\Sigma(L)$ for an alternating knot or link $L$. Since a nontrivial $L$-space knot cannot admit both positive and negative $L$-space surgeries and
\[-S_{r}^3(K)=S_{-r}^3(\overline{K})=\Sigma(\overline{L}),\]
we may assume that $p/q>0$ and that all other alternating surgeries on $K$ arise from positive slopes.

Let $D$ be a reduced alternating diagram of $L$. By Theorem~\ref{thm:Gibbonssouped} and Theorem~\ref{thm:altboundssharp}, the lattice $\Lambda_D$ is isomorphic to a $p/q$-changemaker lattice,
\[\Lambda_{\frac{p}{q}}=\langle w_0, \dotsc, w_l \rangle^{\bot}\subseteq \mathbb{Z}^{\rk \Lambda_D +l +1},\]
whose stable coefficients are determined by the Alexander polynomial of $K$.

Since $\Lambda_D$ is the graph lattice associated to the white graph of $D$, $\Lambda_{\frac{p}{q}}$ admits an obtuse superbase. If we write $w_0$ in the form
\[w_0=
\begin{cases}
e_0+\sigma_1 f_1 + \dotsb + \sigma_t f_t    & q>1,\\
\sigma_1 f_1 + \dotsb + \sigma_t f_t        & q=1.
\end{cases}
\]
Since $D$ is reduced, $\Gamma_D$ contains no cut-edges. This implies that $\Lambda_D$ contains no vectors of norm 1 and so $\sigma_i \geq 1$ for all $i$. As we are assuming that $g(K)>0$, \eqref{eq:calculateg} implies that $\sigma_t>1$. So stable coefficients form a nonempty tuple, $(\sigma_m, \dots, \sigma_t)$. This allows us to define \begin{equation*}
N=\sigma_m + \sum_{i=m}^t \sigma_i^2,
\end{equation*}
which will be the integer appearing in the statement of Theorem~\ref{thm:widthbound}.

\begin{proof}[Proof of Theorem~\ref{thm:upperbound} and Theorem~\ref{thm:widthbound}]
Lemma~\ref{prop:nearby} implies that the $\lceil p/q \rceil$-changemaker lattice
\[
\Lambda' =\langle w_0 \rangle^\bot\subseteq
\begin{cases}
\mathbb{Z}^{t+1} & q>1 \\
\mathbb{Z}^{t  } & q=1,
\end{cases}
\]
also admits an obtuse superbase.
As shown by \eqref{eq:calculateg}, we have \[2g(K)=\sum_{i=1}^t \sigma_i (\sigma_i-1).\] Therefore, Lemma~\ref{lem:gbound} gives the bound
\[\lceil p/q \rceil = \norm{w_0} \leq 4g(K)+3.\]
This proves Theorem~\ref{thm:upperbound}. From Lemma~\ref{lem:graphbound}, we get the upper bound
\[p/q\leq \norm{w_0}\leq 1+\sigma_m + \sum_{i=m}^t \sigma_i^2=N+1.\]
Since $(\sigma_1, \dotsc, \sigma_t)$ satisfies the changemaker condition, we must have
\[\sigma_m \leq 1+\sum_{i=1}^{m-1}\sigma_i=1+\sum_{i=1}^{m-1}\sigma_i^2,\]
where the second inequality holds since $\sigma_i=1$ for $1\leq i<m$. Thus we obtain
\[p/q\geq \sum_{i=1}^t \sigma_i^2 \geq \sum_{i=m}^t \sigma_i^2 + \sigma_m-1= N-1.\]
This completes the proof of Theorem~\ref{thm:widthbound}.
\end{proof}
The lower bound $N-1$ appearing in this proof arises from the fact that there can be no $r$-changemaker lattice for any $r<N-1$ with stable coefficients $(\sigma_m, \dots, \sigma_t)$. Thus it follows from Theorem~\ref{thm:Gibbonssouped} that if $S^3_r(K)$ bounds a negative-definite sharp manifold for $r>0$, then $r\geq N-1$. This justifies the claim made in Remark~\ref{rem:sharplowerbound}.

\paragraph{}Now it remains to prove Theorem~\ref{thm:fullrange}.
\begin{proof}[Proof of Theorem~\ref{thm:fullrange}]
Assume that $S_{r}^3(K)$ is an alternating surgery for $r\in \{r_1,N,r_2\}$ with $N-1\leq r_1<N <r_2 <N+1$. Let $S_{r_i}^3(K)=\Sigma(L_i)$ for $i=1,2$ and $S_{N}^3(K)=\Sigma(L)$ for $L$ and $L_i$ alternating. For $i=1,2$ let $D_i$ be a reduced alternating diagram for $L_i$ and let $D$ be a reduced alternating diagram for $L$. Theorem~\ref{thm:Gibbonssouped} shows that there is $w_0=\sigma_t f_t+ \dotsb +  \sigma_2 f_2$, such that $\Lambda_{D_1}$ is isomorphic to the $r_1$-changemaker lattice
\[\Lambda_{r_1}=\langle w_0+e_0, w_1, \dotsc, w_{l_1} \rangle^\bot \subseteq \langle f_2, \dotsb, f_t,e_0, \dotsc, e_{s_1} \rangle,\]
$\Lambda_{D_2}$ is isomorphic to the $r_2$-changemaker lattice
\[\Lambda_{r_2}=\langle w_0+f_1+e_0, w_1, \dotsc, w_{l_2} \rangle^\bot \subseteq \langle f_1, f_2, \dotsb, f_t,e_0, \dotsc, e_{s_2} \rangle,\]
and $\Lambda_D$ is isomorphic the $N$-changemaker lattice
\[\Lambda_{N} = \langle w_0 \rangle^\bot \subseteq \langle f_1, \dotsc , f_t \rangle.\]
Since $\Lambda_{r_2}$ admits an obtuse superbase, Proposition~\ref{prop:nearby} implies that $\Lambda_N$ admits an obtuse superbase containing a vertex $v$ with $v\cdot f_1=-2$. Since $\Lambda_{r_1}$ is a changemaker lattice, $(\sigma_2, \dotsc, \sigma_t)$ must satisfy the changemaker condition. Therefore, if $g>1$ is minimal such that $v\cdot f_g \geq 0$, then Proposition~\ref{prop:CMprop} implies that there is $A\subseteq \{2, \dotsc, g-1\}$ with $\sigma_g -1 =\sum_{i\in A}\sigma_i$. If we set $z=f_g -f_1 - \sum_{i\in A} f_i$, we have $z\in \Lambda_N$ and we can compute
\begin{align*}
(v-z) \cdot z &= v\cdot f_g - 1 + -(v\cdot f_1 + 1) - \sum_{i \in A} (v\cdot f_i +1)\\
    &\geq v\cdot f_g - v\cdot f_1 -2 = v\cdot f_g\geq 0.
\end{align*}
Since $z\ne v$, this shows that $v$ is reducible. Thus Lemma~\ref{lem:2connectgraphlat} implies that $\Lambda_{N}$ is decomposable and that if $\Lambda_N$ is isomorphic to a graph lattice $\Lambda(G)$ for any connected graph $G$, then $G$ contains a cut vertex. This shows that the white graph $\Gamma_D$ contains a cut vertex. Since, we have assumed that $D$ is reduced, this implies that $L=L_1 \# L_2$ for nontrivial $L_1$ and $L_2$. Therefore $S_{N}^3(K)=\Sigma(L_1)\# \Sigma(L_2)$ is reducible. Using work of Hoffman, Matignon-Sayari showed that if $S_N^3(K)$ is a reducible surgery, then either $N\leq 2g(K)-1$ or $K$ is a cable knot \cite{hoffman98reducing,Matignon03longitudinal}. Since we have
\[N>2g(K)=\sum_{i=1}^t \sigma_i(\sigma_i-1),\]
it follows that $K$ is cable knot. This completes the proof of Theorem~\ref{thm:fullrange}.
\end{proof}

\section{Examples and questions}\label{sec:examples}
We give some examples relating to alternating surgeries and sharp 4-manifolds to illustrate the results of this paper. We then conclude the paper by discussing some questions that arise naturally from this work.
%In Section~\ref{sec:Montytrick}, we show how the Montesinos trick can be used to produce knots admitting alternating surgeries. After that, we discuss the alternating surgeries on $(-2,3,7)$-pretzel knot. We also give examples of $L$-space knots which don't admit any alternating surgeries. These examples are all cables of the trefoil, however we must use the results of this paper in different ways in each case.

\subsection{Alternating surgeries via the Montesinos trick}\label{sec:Montytrick}
We will now describe a construction for building knots admitting alternating surgeries. As far as the author is aware, this construction accounts for all known examples of alternating surgeries.

\paragraph{}An {\em almost-alternating diagram} $D$ is one which can be obtained by a crossing change from an alternating diagram. We call a crossing which can be changed to obtain an alternating diagram a {\em dealternating crossing}. Now let $D$ be an almost-alternating diagram of the unknot with a dealternating crossing $c$ and let $B$ be a small ball containing $c$. Since the double cover of $S^3$ branched over the unknot is $S^3$, the ball $B$ lifts to a solid torus $T\subseteq S^3$ when we take the double cover of $S^3$ branched over $D$. Let $K\subseteq S^3$ be the knot given by the core of $T$. If $D'$ is obtained from $D$ by replacing $c$ with some other rational tangle, then the Montesinos trick shows that $\Sigma(D')$ is obtained by surgery on $K$ \cite{montesinos1973variedades}. Since we may perform tangle replacements such that the resulting diagram is alternating, we see that $K$ admits alternating surgeries. If we take $D'$ to be the alternating diagram obtained by changing $c$, then the resulting surgery is half-integral
\[S^3_{n+\frac{1}{2}}(K)=\Sigma(D'),\]
for some $n\in \mathbb{Z}$. By reflecting $D$, if necessary, we may assume that $n$ is positive. It can be shown (e.g \cite[Proposition~5.4]{mccoy2014noninteger}) there are tangle replacements showing that $S^3_{r}(K)$ is an alternating surgery for all $r$ in the range $n\leq r \leq n+1$.
\begin{rem}\label{rem:Liamsremark}
It follows from the work of Watson that for all $r\geq n$, the manifold $S^3_{r}(K)$ is the double branched cover of a quasi-alternating link $L$ \cite{watson11quasi}. However, Theorem~\ref{thm:widthbound} shows that when $K$ is non-trivial $L$ can only be alternating for $r\leq n+2$. Thus we see that almost-alternating diagrams of the unknot gives rise to infinite families of non-alternating quasi-alternating knots and links.
\end{rem}
\begin{rem} It follows from Theorem~\ref{thm:fullrange}, that if $K$ is not a cable knot or the unknot, then $K$ can admit at most one other alternating surgery with $r=n+2$ or $r=n-1$. If one uses the generalisation of Theorem~\ref{thm:fullrange} asserted in Remark~\ref{rem:fullrangeextension}, then we see that actually neither of these possibilities can arise and that $S_{r}^3(K)$ is an alternating surgery if and only if $n\leq r\leq n+1$.
\end{rem}

As an example, we see what the results of this paper say about alternating surgeries on the $(-2,3,7)$-pretzel knot and describe how they arise through the construction given in this section.

\begin{exam}\label{ex:pretzel}
Let $K$ denote the $(-2,3,7)$-pretzel knot. It is well-known that $K$ admits two lens space surgeries \cite{fintushel80lenssurgery}. This implies that $K$ is an $L$-space knot and in particular that it has alternating surgeries. The Alexander polynomial is
\[\Delta_K(t)=t^5+t^{-5}-(t^4+t^{-4})+ t^2+t^{-2} -(t^1+t^{-1}) + 1.\]
The corresponding non-zero torsion coefficients are $t_0=t_1=2$ and $t_2=t_3=t_4=1$. From Lemma~\ref{lem:mubound} we can deduce that the stable coefficients of the corresponding changemaker vector are $(2,2,3)$. If we apply Theorem~\ref{thm:widthbound} to $K$, then integer $N$ we obtain is $N=3^2+2^2+2^2 + 2=19$. Therefore, if $S_r^3(K)$ is an alternating surgery, then $18\leq r \leq 20$.

Since the changemaker lattice
\[L=\langle 3f_6 + 2f_5 + 2f_4 + f_3 + f_2 + f_1 \rangle^\bot\]
does not admit an obtuse superbase, we see that $S^3_r(K)$ cannot be an alternating surgery for $19<r \leq 20$.

In fact, $K$ arises through the construction given in Section~\ref{sec:Montytrick}, and for each $r$ in $18\leq r\leq 19$, $S_r^3(K)$ branches over an alternating knot or link obtained by tangle replacement on the knot $8_{17}$, as shown in Figure~\ref{fig:pretzelexample}.
\end{exam}

\begin{figure}[h]
  \centering
  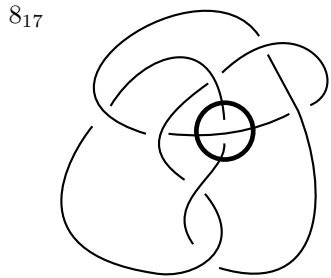
 \caption{A diagram of $8_{17}$ with its unknotting crossing circled. For each $r$ in the range $18\leq r \leq 19$, $r$-surgery on the $(-2,3,7)$-pretzel knot yields the branched double cover of an alternating knot or link obtained by replacing the unknotting crossing in $8_{17}$ by some rational tangle. Note that both resolutions of the unknotting crossing gives a 2-bridge knot or link. The two resolutions correspond to the cases $r=18$ and 19.}
 \label{fig:pretzelexample}
\end{figure}

%\subsection{Torus knots}
%The torus knot $T_{m,n}$ admits alternating surgeries for all slopes in the range $mn-1\leq r \leq mn+1$. This can be seen by explicitly calculating the manifold arising by surgery on the torus knot (e.g Lemma~4.4 of \cite{Owens12negdef}) and showing that for this range they are the branched double cover of alternating links, (which are necessarily 2-bridge or Montesinos links). Alternatively, we can see that these surgeries arise through the construction in Section~\ref{sec:Montytrick}. The manifold $S^3_{mn\pm \frac{1}{2}}(T_{m,n})$ is the lens space $L(2mn \pm 1, 2n^2)$ \cite{Moser71elementary}. Since $L(2mn \pm 1, 2n^2)$ is the double branched cover of an unknotting number one 2-bridge knot, we can take an alternating 2-bridge diagram for this knot and by changing an appropriate choice of unknotting crossing, we obtain an almost-alternating diagram of the unknot for which the construction in Section~\ref{sec:Montytrick} yields $T_{m,n}$.
\subsection{Some knots with no alternating surgeries}
We  use  the  results  of  this paper to exhibit two examples of
$L$-space knots which do not admit any alternating surgeries. Although both are cables of the trefoil, they do not admit alternating surgeries for different reasons: in one case, the cabling slope is ``too large" and
in the other it is ``too small".

\begin{exam}\label{ex:215cable}
Let $K$ be the $(2,15)$-cable of $T_{2,3}$. Since $S^3_{30}(K)= S^3_{15/2}(T_{3,2})\# L(2,1)$ is an $L$-space, $K$ is an $L$-space knot. We will show that this does not admit any alternating surgeries. The Alexander polynomial of $K$ is given by
\[\Delta_K(t)=t^9+t^{-9}-(t^8+t^{-8}) +t^5+t^{-5}-(t^4+t^{-4}) +t^3+t^{-3}-(t^2+t^{-2})+t+t^{-1}-1.\]
By the observations of Remark~\ref{rem:LspaceTi} and equation \eqref{eq:calculateg}, we see that the stable coefficients given by $K$ must be $(2,2,2,4)$. Thus the quantity $N$ in Theorem~\ref{thm:widthbound} is given by $N=30$. Combining this with Proposition~\ref{prop:nearby}, we see that to verify that $K$ has no alternating surgeries we need only check that none of the
three changemaker lattices
\[L_{29}=\langle 4e_0+2e_1+2e_2+2e_3+e_4\rangle^\bot,\]
\[L_{30}=\langle 4e_0+2e_1+2e_2+2e_3+e_4+e_5\rangle^\bot,\]
or
\[L_{31}=\langle 4e_0+2e_1+2e_2+2e_3+e_4+e_5+e_6\rangle^\bot\]
admit obtuse superbases. Since this can be verified relatively easily, for example by using that in each case there are only a small number of irreducible vectors $v$ with $v\cdot e_0 \ne 0$, we see that $K$ does not admit any alternating surgeries.
\end{exam}
\begin{exam}\label{ex:23cable}
Let $K$ be the $(2,3)$-cable of $T_{2,3}$. We will show that this is an $L$-space knot not admitting any alternating surgeries. The Alexander polynomial of $K$ is
\[\Delta_K(t)=t^3 -t^2+1-t^{-2}+t^{-3}.\]
Observe that this is the same as the Alexander polynomial for the torus knot $T_{3,4}$. If $S^3_{r}(K)=\Sigma(L)$ were an alternating surgery, then for any reduced alternating diagram $D$ of $L$, the white lattice $\Lambda_D$ is isomorphic to an $r$-changemaker lattice with stable coefficients given by $(3)$. It follows that we must have $11\leq r \leq 13$. Since $S^3_{r}(T_{4,3})$ is an alternating surgery for any $r$ in this range, we must have $\Lambda_D \cong \Lambda_{D'}$, where $D'$ and is any reduced alternating diagram for an alternating knot or link $L'$ such that $\Sigma(L')=S^3_{r}(T_{4,3})$. Since $L$ and $L'$ are alternating, this isomorphism of white lattices implies that $L$ and $L'$ must be mutants of one another and that $\Sigma(L)=\Sigma(L')=S^3_{r}(T_{4,3})$ \cite{Greene13Mutation}. Surgery on a torus knot is always a small Seifert fibred space \cite{Moser71elementary}, but $S^3_{r}(K)$ is a small Seifert fibred space only if $r$ takes the form $r=6\pm\frac{1}{q}$ \cite{Gordon83Satellite}. Thus $K$ admits no alternating surgeries.
\end{exam}

\subsection{Surgeries bounding sharp 4-manifold}
It seems natural to wonder what we can say about the set of positive surgery slopes for which a given knot bounds a negative-definite sharp manifold. It can be shown that if it is non-empty then this set is an unbounded interval.
\begin{thm}[\cite{mccoy2014alexpoly}, Theorem~1.2]\label{thm:sharpupwards}
Let $K$ be a knot in $S^3$. If $S_{p/q}^3(K)$ bounds a sharp negative-definite 4-manifold for some $p/q>0$, then $S_{p'/q'}^3(K)$ bounds a sharp negative-definite 4-manifold for all $p'/q'\geq p/q$.\qed
\end{thm}
This allows us to characterise the set of all such slopes for torus knots admitting positive $L$-space surgeries.
\begin{prop}
For $r,s>1$ and $p/q>0$, the manifold $S^3_{p/q}(T_{r,s})$ bounds a negative-definite sharp 4-manifold if and only if $p/q\geq rs-1$.
\end{prop}
\begin{proof}
Since $S^3_{rs-1}(T_{r,s})$ is a lens space \cite{Moser71elementary}, Theorem~\ref{thm:sharpupwards} shows that $S^3_{p/q}(T_{r,s})$ bounds a negative-definite sharp 4-manifold for any $p/q\geq rs-1$. To obtain the converse, observe that $S^3_{rs+1}(T_{r,s})$ is also a lens space and hence also an alternating surgery. Thus for $K=T_{r,s}$, we see that the integer $N$ in Theorem~\ref{thm:widthbound} is $N=rs$. Thus Remark~\ref{rem:sharplowerbound} gives the desired lower bound.
\end{proof}
There are also examples of $L$-space knots for which no such slopes exist.
\begin{exam}\label{ex:25cable}
Let $K$ be the $(2,5)$-cable of $T_{2,3}$. We will show that $K$ is an $L$-space knot such that $S^3_r(K)$ cannot bound a sharp negative-definite 4-manifold for any $r>0$. Since $S^3_{10}(K)= S^3_{5/2}(T_{3,2})\# L(2,1)$ is an $L$-space, $K$ is an $L$-space knot.
To show that $S^3_r(K)$ cannot bound a sharp 4-manifold, we show there is no vector satisfying \eqref{eq:CMformula}.
 The Alexander polynomial of $K$ is
\[\Delta_K(t)=t^4 -t^3+1-t^{-3}+t^{-4},\]
which has non-zero torsion coefficients $t_0(K)=t_1(K)=t_2(K)=t_3(K)=1$. Thus by Remark~\ref{rem:LspaceTi}, we can assume that the first coordinate of any vector satisfying \eqref{eq:CMformula} is $\sigma_0=4$. However this contradicts \eqref{eq:calculateg}, which implies that we must have $\sigma_0(\sigma_0-1)\leq 2g(K)=8$.
\end{exam}

\subsection{Further questions}
Given the results of this paper, it is natural to wonder how the set of knots admitting alternating surgeries are contained within the set of all $L$-space knots. For the purposes of this discussion we define several classes of $L$-space knots. We will restrict our attention to those admitting positive $L$-space surgeries. We say that $S^3_r(K)$ is a {\em quasi-alternating surgery} if it is the double branched cover of a quasi-alternating knot or link.
\begin{align*}
\mathcal{L}&=\{K:\exists r>0 \text{ such that } S^3_r(K)\text{ is an $L$-space.}\}\\
\mathcal{A}&=\{K:\exists r>0 \text{ such that } S^3_r(K) \text{ is an alternating surgery.}\}\\
\mathcal{D}&=\{K: \text{$K$ is branched double-cover of an unknotting arc in an} \\
            & \hspace{70pt} \text{alternating diagram.}\}\\
\mathcal{QA}&=\{K: \exists  r>0 \text{ such that } S^3_r(K) \text{ is a quasi-alternating surgery.}\}
\end{align*}
Since the double branched cover of a quasi-alternating knot is an $L$-space and any alternating link is quasi-alternating, these sets satisfy the following inclusions:
\[\mathcal{D}\subseteq\mathcal{A}\subseteq\mathcal{QA}\subseteq\mathcal{L}.\]
Watson has shown that any sufficiently large cable of a torus knot is in $\mathcal{QA}$ \cite{watson11quasi}. In particular, the $(2,15)$-cable of $T(2,3)$ is in $\mathcal{QA}$. As we have shown that it is not in $\mathcal{A}$, this shows that $\mathcal{A} \subsetneq \mathcal{QA}$.
\begin{rem} It seems probable that there are $L$-space knots which do not admit quasi-alternating surgeries. The $(2,3)$-cable of $T_{2,3}$ and the $(2,5)$-cable of $T_{2,3}$ seem to be potential candidates for knots in $\mathcal{L} \setminus \mathcal{QA}$.
\end{rem}
As far as the author is aware, all known examples of knots in $\mathcal{A}$ are also in $\mathcal{D}$, i.e they arise through the construction in Section~\ref{sec:Montytrick}. Moreover, it is known for every non-integer alternating surgeries, there is a knot in $\mathcal{D}$ with the same surgery.
\begin{thm}[\cite{mccoy2014noninteger}, Theorem~1.2]\label{thm:nonintegersurgery}
If $S_{p/q}^3(K)$ is an alternating surgery for with $q>1$, then there is $K'\in \mathcal{D}$ with $S_{p/q}^3(K)=S_{p/q}^3(K')$.\qed
\end{thm}
This suggests the following conjecture.
\begin{conj}\label{conj:wild}
Every alternating surgery arises as tangle replacement on an almost-alternating diagram of the unknot, that is, we have $\mathcal{A}=\mathcal{D}$.
\end{conj}
Since lens spaces arise as the double branched covers of alternating links, one can ask how this conjecture agrees with results and conjectures on lens space surgeries. The cyclic surgery theorem of Culler, Gordon, Luecke and Shalen shows that only torus knots admit non-integer lens space surgeries \cite{cglscyclic}. Since torus knots are in $\mathcal{D}$, this verifies Conjecture~\ref{conj:wild} in certain cases.

Short of attacking Conjecture~\ref{conj:wild} in full, there are various related questions we can ask.
\begin{ques}
Does Theorem~\ref{thm:nonintegersurgery} extend to the case of integer alternating surgeries?
\end{ques}
It follows from their construction that every knot in $\mathcal{D}$ admits a strong inversion.
\begin{ques}
Is every knot in $\mathcal{A}$ strongly invertible?
\end{ques}
It seems likely that any progress on Conjecture~\ref{conj:wild} would require an alternative description of the class $\mathcal{D}$.
\begin{ques}
Is there a characterisation of $\mathcal{D}$ which does not refer to almost-alternating diagrams of the unknot?
\end{ques}
\paragraph{} Finally, as we demonstrated with the $(2,5)$-cable of $T_{2,3}$, equation \eqref{eq:CMformula} can be used to show that for some knots no manifold obtained by positive surgery can bound a negative-definite sharp manifold. As we saw in Example~\ref{ex:23cable}, the $(2,3)$-cable of $T_{2,3}$ passes this obstruction as it has the same Alexander polynomial as $T_{3,4}$. However, it seems unlikely that any positive surgery on the $(2,3)$-cable of $T_{2,3}$ bounds a sharp manifold.
\begin{ques}
Can one find alternative ways to show that surgery on a knot does not bound a sharp 4-manifold? In particular, is it possible to show that no positive surgery on the $(2,3)$-cable of $T_{2,3}$ bounds a sharp manifold?
\end{ques}
\bibliographystyle{plain}
\bibliography{alternatingsurgeries}
\end{document}